\documentclass[12pt]{amsart}
\usepackage{latexsym,amssymb,amsfonts,amsmath, amscd, euscript, MnSymbol}
\usepackage{pdfsync}
\usepackage{enumerate}
\usepackage[pdftex]{graphicx}
\usepackage{wrapfig}
\usepackage{epsfig}
\usepackage{pgf,tikz}
\usetikzlibrary{arrows,automata}
\usetikzlibrary{positioning,calc}
\usepackage{amssymb, amsmath, amsthm, enumerate}
\usepackage{mathrsfs}
\usepackage{url}
\usepackage{epstopdf}
\usepackage{enumerate}
\usepackage{verbatim}
\usepackage{float}
\usepackage{layout}
\usepackage[top=3cm, bottom=3cm, left=2cm, right=2cm]{geometry}
\usepackage{mathtools}

\newtheorem{theorem}{Theorem}[section]

\newtheorem{corollary}[theorem]{Corollary}

\newtheorem{definition}[theorem]{Definition}
\newtheorem{example}[theorem]{Example}
\newtheorem{lemma}{Lemma}[section]

\newtheorem{proposition}[theorem]{Proposition}

\newtheorem{remark}[theorem]{Remark}

\newtheorem{question}[theorem]{Question}
\usepackage{siunitx}
\usepackage{graphics,mwe}

\newsavebox{\hold}
\newlength{\holdht}

\begin{document}

\title[Best proximity points in ultrametric spaces]{Best proximity points in ultrametric spaces}
\author[K. Chaira and S. Lazaiz]{K. Chaira$^1$ and S. Lazaiz$^{2}$}
\address{$^1$L3A Laboratory\\ Department of Mathematics and Computer Sciences\\ Faculty of Sciences Ben M'sik\\ University of Hassan II Casablanca\\}
\email{chaira\_karim@yahoo.fr}
\address{$^2$LaSMA Laboratory\\ Department of Mathematics\\ Faculty of Sciences Dhar El Mahraz\\ University Sidi Mohamed Ben Abdellah, Fes, Morocco\\}
\email{samih.lazaiz@usmba.ac.ma}


\begin{abstract}
In the present paper, we study the existence of best proximity pair in ultrametric spaces. We show, under suitable assumptions, that the proximinal pair $(A,B)$ has a best proximity pair. As a consequence we generalize a well known best approximation result and we derive some fixed point theorems. Moreover, we provide examples to illustrate the obtained results.

\end{abstract}

\maketitle

\section{Introduction and preliminaries}

Let $T:A\rightarrow B$ be a map where $A$ and $B$ are two nonempty subsets of a metric space $M$. Recall that a point $x\in M$ is called a fixed point of a $T$ if $Tx=x$. It is known that such an equation does not always have a solution. However, in the absence of the fixed point (for example if $A\cap B=\varnothing$), it is possible to consider the problem of finding a point $x\in A$ that is as close as possible to $Tx$ in $B$; i.e., to minimize the quantity $d(x,Tx)$ over $A$. This type of problem is considered as part of approximation theory, more specifically best approximation point results. 

\begin{definition}\cite{ref12}
A subset $A\subset M$ is said to be proximinal if given any $x\in M$ there exists $a_0\in A$ such that 
$$d(x,a_0)=dist(x,A)=\inf\{d(x,z):\; z\in A\}.$$
Such an $a_0$, if it exists, is called a best approximation to $x$ in $A$.
\end{definition}

In the literature, some positive results concerning the existence of best approximation point were given whether in the archimedean or non-archimedean spaces, for more details see \cite{ref9,ref10,ref11}.

\medskip 
However, proximinality results guaranteed the existence of best approximation point without being optimal, for this reason the authors in \cite{ref12} introduce the notion of best proximity pair which gives both the existence and the optimality. Recall that for a bounded set $B$ in $M$, $\delta(B)$ stands for the diameter of $B$, i.e., 
\[
\delta(B):=\sup\{d(x,y)\; : \;x,y\in B \},
\]
and $dist(A,B)=\inf\{d(a,b):\;a\in A\;\text{and}\;b\in B\}$.

\begin{definition}\cite[Definition 1.1.]{ref12}
Let $M$ be a metric space and let $A$ and $B$ be nonempty subsets of $M$. Let
$$A_0=\{x\in A:\; d(x, y)=dist(A, B)\; \text{for some}\; y\in B\};$$
$$B_0=\{y\in B:\; d(x, y)=dist(A, B)\; \text{for some}\; x\in A\}.$$

A pair $(x, y)\in A_0 \times B_0$ for which $d(x, y)=dist(A, B)$ is called a best proximity pair for $A$ and $B$.
\end{definition} 

Motivated by the preceding definition, the authors in \cite{ref12} raised the following question:

\begin{question}\label{Q1}
Let $(A,B)$ be proximinal pair of $(M,d)$. Does there exists a best proximity pair $(a,b)\in A_0\times B_0$ ? If so, does the pair $(A_0,B_0)$ also proximinal ?
\end{question}

If this question has an affirmative answer, this inspires to formulate the following:

\begin{question}\label{Q2}
Given a mapping $T:A\cup B \rightarrow A\cup B$ with $T(A)\subset A$ and $T(B)\subset B$ (also called noncyclic mapping), does there exists an ordered pair $(a,b)\in A\times B$ such that 
$$Ta=a, \quad Tb=b\quad \text{and}\quad d(a,b)=dist(A,B).$$
\end{question}

There is an extensive literature contains partial affirmative answers to these two questions in the context of metric spaces and linear spaces (see \cite{ref13,ref14,ref15}). To the best knowledge of the authors, this is the first time these questions are considered in the case of ultrametric spaces. 

\medskip
Recall that an ultrametric space is a metric space $M$ in which strong triangle inequality viz. $d(x,y) \leq \max\{d(x,z),d(z, y)\}$ is satisfied for all $x,y,z \in M$. 

\begin{definition}\cite{ref22}
An ultrametric space $(M,d)$ is called spherically complete if each nested sequence of balls $B_1 \supset  B_2 \supset\cdots$  has a non-empty intersection.
\end{definition}

In \cite{ref8} the authors prove the following result.

\begin{theorem}\label{KS1}
  Let $A$ be a spherically complete subspace of an ultrametric space $(M,d)$. Then $A$ is proximinal in $M$.
\end{theorem}
\begin{definition}\cite{ref3}
Assume that $T:M\rightarrow M$ is a map and $B=B(x,r)\subset M$ a closed ball of $M$. We say that $B$ is \textit{minimal $T$-invariant ball} if : 
\begin{itemize}
  \item[(i)] $T(B)\subseteq B$, and
  \item[(ii)] $d(y,Ty)=r$ for each $y\in B$. 
\end{itemize}
\end{definition}

The following definition is well known.

\begin{definition}
Let $A\subseteq M$. $T:A\rightarrow A$ is said to be nonexpansive if  
\[
d(Tx,Ty)\leq d(x,y)\quad \text{for each}\; x,y\in A
\]
\end{definition}

The next fixed point theorem is the most important and significant result in ultrametric spaces.

\begin{theorem}\label{KS}\cite{ref3,ref2}
Suppose $M$ is a spherically complete ultrametric space and $T:M\rightarrow M$ is a nonexpansive map. Then every ball of the form $B(x,d (x,Tx))$ contains either a fixed point of $T$ or a minimal $T$-invariant ball. 
\end{theorem}

Note that the proof provided of this Theorem, shows that the fixed point for nonexpansive mappings is closely related with the lack of the condition $d(y,Ty) = d(z,Tz) > 0$, for any $y$ and $z$ in a minimal $T$-invariant ball. That is to say, this condition must be violated. Thereby, other conditions must be added in order to ensure the existence of fixed point. 

\begin{definition}
Let $(M,d)$ be an ultrametric space and $T : M \rightarrow M$ a mapping. We say that :
\begin{enumerate}
  \item $T$ is strictly contractive if $d(Tx, Ty) < d ( x , y )$ whenever $x\neq y$.
  \item $T$ is strictly contractive on orbit if $Tx\neq x$ implies $d( T^2x, Tx) < d ( Tx, x )$ for each $x \in M$.
\end{enumerate}
\end{definition} 

It is clear that if $T$ is strictly contractive then it is strictly contractive on orbit. On the other hand, the authors in \cite{ref18} introduced a more general class of mappings that contains the preceding one.

\begin{definition}
Let $(M,d)$ be an ultrametric space. We will say that $T$ has the weak-regular property if
$$\limsup_{n\rightarrow \infty}\ d(T^nx,T^{n+1}x)< d(x,Tx),$$
for each $x$ in $M$ such that $x \neq Tx$.
\end{definition}

Using this definition, the authors in \cite{ref18} proved the following result.

\begin{theorem}\label{w-regular}
Let $(M,d)$ be a spherically complete ultrametric space.  Let $T:M\rightarrow M$ be a nonexpansive map which has the weak-regular property. Then $T$ has a fixed point in any $T$-invariant closed ball.
\end{theorem}

\medskip
In this paper, we will see that the answer to the above questions is positive under natural assumptions. In the lack of this assumptions, we show by an example that $(A_0,B_0)$ may be an empty pair. As a consequence we generalize a best approximation result due to Kirk and Shahzad (see \cite[Theorem 11]{ref3}) and we derive some fixed point theorems. Throughout this paper, we provide some examples to illustrate the obtained results.

\section{Main results}

We first give a useful lemma.

\begin{lemma}\label{Lemma1}
Let $(A,B)$ be a proximinal pair of an ultrametric space $(M,d)$ with $B$ is bounded. Assume that $\delta(B)\leq dist(A,B)$, then $A_0\neq\varnothing$ and $B_0=B$. 
\end{lemma}

\begin{proof}
\begin{itemize}
  \item Since $\delta(B)\leq dist(A,B)$ we have for all $b,b^\prime\in B$ and $a\in A$ 

$$d(b,b^\prime)\leq d(a,b^\prime).$$

If there exist $b,b^\prime\in B$ and $a\in A$ such that $d(b,b^\prime)=d(a,b^\prime)$, then $d(a,b^\prime)=dist(A,B)$ which implies that $A_0\neq\varnothing$ and $B_0\neq\varnothing$.

\medskip
Assume now that for all $b,b^\prime\in B$ and all $a\in A$ we have $d(b,b^\prime)< d(a,b^\prime)$, then 
$$d(a,b^\prime)=d(a,b).$$

Thus,
 $$dist(b,A)=dist(b^\prime, A),$$
 
 which implies that $\varphi(b)=dist(b,A)$ is a constant mapping over $B$. Thus,
 $$\inf_{b\in B} \varphi (b)=\varphi(b)$$
 that is $dist(A,B)=dist(b,A)=d(b,a)$ since $A$ is proximinal.  Then $(A_0,B_0)$ is a nonempty pair.

  \item $B_0=B$. Indeed, let $b\in B$ and $a_0\in A_0$ then there exists $b_0\in B_0$ such that 
   $$d(a_0,b_0)=d$$
   using the strong inequality we get :
   $$
   d(a_0,b)\leq \max\{d(a_0,b_0), d(b_0,b)\}=dist(A,B)
   $$ 
   then $b\in B_0$.

\end{itemize}
\end{proof}

Next, we give an example to show that if $\delta(B)>dist(A,B)$, the pair $(A_0,B_0)$ may be an empty pair.

\begin{example}
Let $M=\mathbb{N}_0$ be the set of positive integers and define the ultrametric distance $d$ on $M$ as follows:
\[
d(n,m)=\begin{cases}
  \begin{array}{ll}
    0 & \text{if}\; n=m\\
    \max\{\frac{1}{n},\frac{1}{m}\} & \text{otherwise.}
  \end{array}
\end{cases}
\]
Then $(M,d)$ is an ultrametric space. Set $A=2\mathbb{N}_0=\{2,4,\ldots\}$ and $B=2\mathbb{N}_0+1=\{1,3,\ldots\}$. 

\medskip
It is clear that $A$ and $B$ are proximinal sets. Moreover, $dist(A,B)=0$ and $\delta(B)=1$ and $\delta(A)=\frac{1}{2}$, thus $dist(A,B)<\min\{\delta(B),\delta(A)\}$. Note that $A_0=B_0=\varnothing$.
\end{example}

The following result gives a positive answer to the question \ref{Q1}.

\begin{proposition}\label{Prop1}
Let $(A,B)$ be a spherically complete pair in an ultrametric space $(M,d)$. Assume that $B$ is bounded and $\delta(B)\leq dist(A,B)$. Then  $(A_0,B_0)$ is nonempty spherically complete  pair of $M$.
\end{proposition}

\begin{proof}
Since every spherically complete set is proximinal (see Theorem \ref{KS1}) we get $A_0\neq \varnothing$ and $B_0=B$ by Lemma \ref{Lemma1}. Thus, obviously $B_0$ is spherically complete. 

\medskip
Next, we show that $A_0$ is spherically complete. Let $\{B(a_i, r_i)\}_{i\in I}$ be a descending collection of closed balls centered at points $a_i \in A_0$. Let $i\in I$, for each $a_i$ in $A_0$ there exists $b_i\in B_0$ such that 
   $$d(a_i,b_i)=d.$$
\begin{description}
     \item[Case 1] Assume that $d\leq r_i$ for all $i\in I$. As any ball is central we have $B(a_i,r_i)=B(b_i,r_i)$ for all $i\in I$, thus 
     $$ B \cap B(a_i,r_i)\neq \varnothing$$
     since $B$ is spherically complete subset we get 
     $$B\cap \bigcap_i B(a_i,r_i)\neq \varnothing.$$
     
     Let $b\in B\cap \bigcap_i B(a_i,r_i)$, then $b\in B$ and 
     $$
     \begin{array}{ccl}
       d(a_i,b)&\leq & \max\{d(a_i,b_i);d(b_i,b)\}\\
               &\leq & \max\{d,\delta(B)\}\\
               &\leq & d
     \end{array}
     $$
     then $b\in B_0$. So, let $a\in A_0$ such that $d(a,b)=d$; thus
     $$
     \begin{array}{ccl}
       d(a,a_i)&\leq & \max\{d(a,b);d(b,a_i)\}\\
               &\leq & d
     \end{array}
     $$
     then $a\in B(a_i,d)\subset B(a_i,r_i)$ for all $i\in I$, i.e.,
     $$a\in A_0\cap \bigcap_i B(a_i,r_i).$$
     \item[Case 2] There exists $i_0\in I$ such that $r_{i_0}\leq d$, then $B(a_{i_0},r_{i_0})\subset B(a_{i_0},d)$. As $A$ is spherically complete, let $a$ be in  $A \cap \bigcap_i B(a_i,r_i)$ and since
     $$A \cap \bigcap_i B(a_i,r_i)\subset A\cap B(a_{i_0},r_{i_0})$$
     we have $d(a,a_{i_0})\leq d$ thus by the strong inequality we get also $d(a,b_{i_0})=d$. And since $b_{i_0}\in B_0$ we have $a\in A_0$ and then 
     $$a\in A_0\cap \bigcap_i B(a_i,r_i).$$
   \end{description}
   Thus, $A_0$ is a spherically complete subset of $M$.   

\end{proof}

The next Theorem answers partially question \ref{Q2} for nonexpansive mappings. 

\begin{theorem}\label{Thm1}
  Let $(A,B)$ be a spherically complete pair of an ultrametric space $M$ for which $B$ is bounded and $\delta(B)\leq dist(A,B)$. Suppose $T : A\cup B\rightarrow A\cup B$ is a noncyclic nonexpansive mapping on $A\cup B$. Then there exists a best proximity pair $(a^*,b^*)\in A\times B$ satisfying one of these statements:
  \begin{enumerate}
    \item[(i)] $a^*$ and $b^*$ are fixed points of $T$.
    
    \item[(ii)] $a^*$ is a fixed point of $T$ and $B(b^*,d(b^*,Tb^*))$ is a minimal $T$-invariant set in $B$, each point of which is a nearest point to $a^*$.
    
    \item[(iii)] $B(a^*,d(a^*,Ta^*))$ (resp. $B(b^*,d(b^*,Tb^*))$) is a minimal $T$-invariant set in $A$ (resp. in $B$).
  \end{enumerate}  
\end{theorem}

\begin{remark}
Let $B_1$ be a ball, $a\in M$, $a\notin B$. Then $d(a, x) = d(a, y)$ for each $x, y \in B$. More generally, let $B_1$, $B_2$ be disjoint balls. Then, for $x \in B_1$ and $y \in B_2$, the number $d(x, y)$ is constant.
\end{remark}

\begin{proof}[Proof of Theorem \ref{Thm1}]
 Using Proposition \ref{Prop1}, we have $(A_0,B_0)$ is a nonempty spherically complete pair of $M$. Let $x\in A_0$, then there exists $y\in B_0$ such that 
$$d(x,y)=dist(A,B)$$
and since $T$ is nonexpansive on $A\cup B$, we have 
$$d(Tx,Ty)\leq d(x,y)=dist(A,B),$$
using the fact that $T$ is noncyclic we get that $T(A_0)\subset A_0$. 

\medskip 
Fix $x\in A_0$ and let $z\in B(x,d(x,Tx))$, then we have 
  $$
  \begin{array}{ccl}
    d(x,Tz) & \leq & \max\{ d(x,Tx); d(Tx,Tz)\}\\
            & \leq &  \max\{ d(x,Tx); d(x,z)\}\\
            & \leq & d(x,Tx),
  \end{array}
  $$ 
  that is $Tz \in B(x,d(x,Tx))$, thus $B(x,d(x,Tx))$ is $T$-invariant subset of $A$. Set $X=B(x,d(x,Tx))\cap A_0$. Using the above we have $T(X) \subset X$. In the same manner we prove for $y\in B_0$ that $Y=B(y,d(y,Ty))\cap B_0$ is a $T$-invariant subset of $B$. Finally, we obtain that $(X,Y)$ is a spherically complete pair of $M$.
  
\medskip
Hence, by Theorem \ref{KS}, one of the following statements holds:
\begin{enumerate}[(i)]

  \item $X$ contains a fixed point $a^*$ of $T$ and $Y$ contains a fixed point $b^*$ of $T$.
  
  \item $X$ contains a fixed point $a^*$ of $T$ and $Y$ contains a minimal $T$-invariant ball.
  
  \item $X$ contains a minimal $T$-invariant ball $B(a^*,d(a^*,Ta^*))$ and $Y$ contains a minimal $T$-invariant ball $B(b^*,d(b^*,Tb^*))$.

\end{enumerate}

Let $u\in B(a^*,d(a^*,Ta^*))$, then $u\in A_0$, hence there exists $v_u\in B_0$ such that 
$$d(u,v_u)=dist(A,B).$$ 
Thus for all $v\in B(b^*,d(b^*,Tb^*))$, we get 
$$
\begin{array}{ccl}
  d(u,v)&\leq & \max\{d(u,v_u),d(v_u,v)\}\\
        &\leq & dist(A,B)
\end{array}
$$
then $d(u,v)=dist(A,B)$ for all $u\in B(a^*,d(a^*,Ta^*))$ and $v\in B(b^*,d(b^*,Tb^*))$. The proof is completed.
\end{proof}

In the next example, we shall consider a mapping $T$ that has fixed points in every ball of the form $B(x,d(x,Tx))$.

\begin{example}\label{EX1}
Let $M$ be the set of all infinite sequences of natural numbers endowed with the ultrametric distance $d$ defined by :

$$d(x,y)=\frac{1}{1+k(x,y)},$$

where $k(x,y)=\inf\{n\in \mathbb{N} :\;x_n\neq y_n\}$.

\medskip 
We define $T:M\rightarrow M$ by $T(x)(0)=x_0$, $T(x)(1)=x_1$ and $T(x)(n)= {\displaystyle \prod_{i=0}^{n} x_{i}}$ for all $n\geq 2$.  

Let $x,y\in M$ and set $d(x,y)=\frac{1}{1+k(x,y)}$, then we have two cases :
\begin{itemize}
  \item If there exists $n$ such that $n<k(x,y)$ and $x_n=0$. We have $y_n=0$, hence, $Tx=Ty$ which implies that $d(Tx,Ty)=0$.
  \item If for all $n<k(x,y)$ we have $x_n\neq 0$, then $x_i=y_i$ for all $0\leq i \leq k(x,y)-1$ 
  so, we have 
  $$Tx=(x_0,x_1,\prod_{i=0}^{2} x_{i},\ldots,\prod_{i=0}^{n} x_{i},\ldots,\prod_{i=0}^{k(x,y)-1} x_{i},\prod_{i=0}^{k(x,y)} x_{i},\prod_{i=0}^{k(x,y)+1} x_{i},\ldots),$$
  and
  $$Ty=(x_0,x_1,\prod_{i=0}^{2} x_{i},\ldots,\prod_{i=0}^{n} x_{i},\ldots,\prod_{i=0}^{k(x,y)-1} x_{i},\prod_{i=0}^{k(x,y)} y_{i},\prod_{i=0}^{k(x,y)+1} y_{i},\ldots)$$
  
  which implies that $d(Tx,Ty)=\frac{1}{1+k(x,y)}$.
\end{itemize}
Thus, in both cases 
$$d(Tx,Ty)\leq d(x,y),$$ 
that is $T$ is nonexpansive. 

\medskip
We shall show that the mapping $T$ satisfies all the hypothesis of Theorem \ref{Thm1}. Denote by $\overline{n}$ the constant sequence $\overline{n}:=(n,n,\ldots)\in M$ and set $X_n=B(\overline{n},\frac{1}{2})$. It is clear that $T(X_n)\subset X_n$ and $d(\overline{n},T\overline{n})=\frac{1}{3}$.

\medskip
Moreover, let $n\neq m$ then, 
$$dist(X_n,X_m)=d(\overline{n},\overline{m})=1$$ 
and 
$$\delta(X_n)=\frac{1}{2},$$ 
which leads to $\delta(X_n)\leq dist(X_n,X_m)$. 

\medskip
So, using Theorem \ref{Thm1}, the mapping $T$ has a fixed point in every set $X_n$. Indeed,  
\begin{enumerate}
  \item if $n=0$ or $n=1$ then $\overline{0}$ respectively $\overline{1}$ is a fixed point of $T$ in $X_0$ respectively in $X_1$;
  \item if $n>1$ then $(n,n,0,0,\ldots)$ which lies in $X_n$ is a fixed point for $T$.  
\end{enumerate}
\end{example}

In the next example, we shall consider a mapping which has stable balls but no fixed points.

\begin{example}
Let $M=\mathbb{Z}_3$ be the ring of the $3$-adic integers endowed with the $3$-adic valuation $|\cdot|_3$. Let $p$ be a fixed larger natural number. Define the pair $(A,B)$ in $\mathbb{Z}_3$ to be $A=B(0,\frac{1}{3^p})$ and $B=B(1,\frac{1}{3^p})$. Note that $A\cap B=\varnothing$. Indeed, if $x\in A\cap B$ then 
$$
|x|_3\leq 1/3^p\qquad\text{and}\qquad|x-1|_3\leq 1/3^p,
$$
so, $1\leq \max\{|1-x|_3,|x|_3\}=\frac{1}{3^p}$ which is a contradiction. Also, it is clear that $B$ is spherically complete whith $\delta(B)=\frac{1}{3^p}$ and $dist(A,B)=1$.

\medskip 
Let $T$ be the mapping defined by $Tx=x+3^p$ for all $x\in M$. It is clear that $T$ is an isometry and noncyclic mapping from $A\cup B$ to itself. Since $T$ does not have any fixed point, by Theorem \ref{Thm1} the pair $(A,B)$ has minimal $T$-invariant pair. Note that, $(A,B)$ is exactly the minimal $T$ invariant pair,  since for all $z\in A\cup B$ we have 
$$|z-Tz|_3=\left|3^p\right|_3=1/3^p.$$
\end{example} 

We consider in the next example a mapping $T$ which has both stable ball and fixed point.

\begin{example}
Let $M$ be the set of all infinite sequences of nonnegative integers endowed with the ultrametric $d$ as Example \ref{EX1}. Let $x=(x_1,x_2,\ldots)\in M$ and define the mapping $T:M\rightarrow M$ by 
$$Tx(n)={\displaystyle \prod_{i=1}^{n} x_{i}},\quad \forall n\geq 1.$$

Let $a^*=(1,1,\ldots)$ and let $b^*=(1,2,2,\ldots)$. Define $A=\{a^*\}$ and $B=B(b^*,d(b^*,Tb^*))$. It is clear that $B$ is spherically complete whith $\delta(B)=\frac{1}{4}$ and $dist(A,B)=\frac{1}{3}$. Moreover, $T$ is nonexpansive and noncyclic mapping from $A\cup B$ to itself. 

\medskip
Note that $a^*$ is a fixed point of $T$ and $B$ is a minimal $T$-invariant ball. Indeed, let $z\in B$ then 
$d(z,b^*)\leq \frac{1}{4}$, that is $k(z,b^*)\geq 3$ (see Example \ref{EX1} for notation), which implies $z=(1,2,z_3,z_4,\ldots)$, thus $Tz=(1,2,2z_3,\ldots)$. Then $d(z,Tz)=\frac{1}{4}$, i.e., $B$ is a minimal $T$-invariant ball.
\end{example}

\begin{corollary}\cite[Theorem 11]{ref3}
Let $A$ be a spherically complete subspace of an ultrametric space $M$, and let $b^* \in M \setminus A$. Suppose $T : M \rightarrow M$ is a mapping for which $Tb^* = b^*$. Also assume that $T$ is nonexpansive on $A \cup \{ b^* \}$ and that $A$ is $T$-invariant. Then $T$ has a fixed point in $A$ which is a nearest point of $b^*$ in $A$, or $A$ contains a minimal $T$-invariant set, each point of which is a nearest point to $b^*$ in $A$.
\end{corollary}

Next, we derive some fixed point results that also answers the question \ref{Q2}.

\begin{theorem}\label{Thm2}
  Let $(A,B)$ be a spherically complete pair of an ultrametric space $M$ for which $B$ is bounded and $\delta(B)\leq dist(A,B)$. Suppose $T : A\cup B\rightarrow A\cup B$ is a noncyclic nonexpansive mapping on $A\cup B$ which has the weak-regular property. Then there exist $a\in A$ and $b\in B$ such that
$$Ta=a, \quad Tb=b\quad \text{and}\quad d(a,b)=dist(A,B).$$
\end{theorem}

\begin{proof}
Using Proposition \ref{Prop1} there exists a best proximity pair $(x,y)\in A_0\times B_0$. Set 
$$X=B(x,d(x,Tx))\cap A_0\quad\text{and}\quad Y=B(y,d(y,Ty))\cap B_0.$$
As the proof of Theorem \ref{Thm1}, we show that $T(X)\subseteq X$ and $T(Y)\subseteq Y$ and since $(X,Y)$ is spherically complete pair we get by theorem \ref{w-regular} the existence of two points $a^*\in X$ and $b^*\in Y$ such that 
$$Ta^*=a^*\quad\text{and}\quad Tb^*=b^*$$
and since $B_0=B$ by Lemma \ref{Lemma1}, we get that $d(a^*,b^*)=dist(A,B)$.
\end{proof}

Since every strictly contractive mapping on orbit has the weak-regular property (see \cite{ref18}), we have the following result.

\begin{corollary}\label{c2}
  Let $(A,B)$ be a spherically complete pair of an ultrametric space $M$ for which $B$ is bounded and $\delta(B)\leq dist(A,B)$. Suppose $T : A\cup B\rightarrow A\cup B$ is a noncyclic nonexpansive mapping on $A\cup B$ which is strictly contractive on orbit. Then there exist $a^*\in A$ and $b^*\in B$ such that
$$Ta^*=a^*, \quad Tb^*=b^*\quad \text{and}\quad d(a^*,b^*)=dist(A,B).$$
\end{corollary}

On the other hand, since every strictly contractive mapping is nonexpansive mapping and the fact that any strictly contractive mapping in a spherically complete ultrametric space has a unique fixed point see \cite{ref19}, we obtain :

\begin{theorem}\label{Thm3}
  Let $(A,B)$ be a spherically complete pair of an ultrametric space $M$ for which $B$ is bounded and $\delta(B)\leq dist(A,B)$. Suppose $T : A\cup B\rightarrow A\cup B$ is a noncyclic strictly contractive mapping on $A\cup B$. Then there exist $a^*\in A$ and $b^*\in B$ such that
$$Ta^*=a^*, \quad Tb^*=b^*\quad \text{and}\quad d(a^*,b^*)=dist(A,B).$$
\end{theorem}

As a consequence of theorem \ref{Thm3} we obtain the following result which can be found in \cite{ref3}.

\begin{corollary}\cite[Corollary 12]{ref3}
Let $A$ be a spherically complete subspace of an ultrametric space $M$ and suppose $T : M \rightarrow M$ is a mapping having a fixed point $b^* \in M \setminus A$. Assume that $T$ is strictly contractive on $A \cup \{ b^*\}$ and $A$ is $T$-invariant. Then there exists a unique fixed point $a^*$ of $T$ which is a nearest point of $b^*$ in $A$.
\end{corollary}

\section{Concluding remarks}

We have studied in this paper the problem of existence of the best proximity point for nonexpansive noncyclic mappings $T:A\cup B\rightarrow A\cup B$. We have shown that if $(A,B)$ is proximinal pair and $B$ is bounded such that $\delta(B)\leq dist(A,B)$, then there exists a proximity pair $(a^*,b^*)\in A_0\times B_0$ such that $d(a^*,Tb^*)=dist(A,B)$. 

\medskip
Now, if one want to deal with cyclic mapping $T:A\cup B\rightarrow A\cup B$ i.e., $T(A)\subseteq B$ and $T(B)\subseteq A$, the result is obvious with the assumption $\delta(B)\leq dist(A,B)$. Indeed, Let $a^*\in A_0$ then there exists $b^*\in B_0$ such that $d(a^*,b^*)=dist(A,B)$. By the strong inequality we get
$$d(a^*,Ta^*)\leq \max\{dist(A,B), d(b^*,Ta^*)\}\leq dist(A,B)$$
thus each element of $A_0$ is a best proximity point of $T$. Note that $T$ is not even assumed to be nonexpansive map.

\medskip
Before discussing another important observation, let's recall the following surprising result.

\begin{proposition}\cite[Proposition 1.2]{ref16}
Let $(M, d)$ be an ultrametric space and let $B$ be a closed ball with radius $r$. Then, the actual radius of $B$ is equal to the diameter of $B$.
\end{proposition}

Let $B$ be a ball. Suppose that $A_0\neq \varnothing$ and  $\delta(B)> dist(A,B)$. Then, there exist $a^*\in A$ and $b^*\in B$ such that 
$$d(a^*,b^*)=dist(A,B).$$

Thus, $d(a^*,b^*)\leq \delta(B)$. By the previous proposition, we obtain $a^*\in B$. Then $A\cap B\neq \varnothing$, hence $dist(A,B)=0$. 

\medskip
Now if one assume that $B$ is no longer a ball we have the following example which shows that $dist(A,B)\neq 0$. This example is borrowed from Kirk and Shahzad in \cite{ref3}.

\begin{example}\label{Ex1}
Let $M = \{ a , b , c , d \}$ with $d ( a , b ) = d ( c , d ) = \frac{1}{2}$; $d ( a , c ) = d ( a , d ) = d ( b , c ) = d ( b , d ) = 1$. Then $( M , d )$ is a spherically complete ultrametric space. Define $A=\{a,c\}$ and $B=\{b,d\}$. Then $(A,B)$ is spherically complete pair with $\delta(B)>dist(A,B)$ and $A_0=\{a,c\}$. Note that $dist(A,B)\neq 0$.
\end{example}


\bibliographystyle{plain}

\end{document}